\documentclass[11pt]{amsart}

\usepackage{amssymb,amsthm,amsmath,epsfig}
\usepackage{verbatim}

\usepackage{color}
\definecolor{darkblue}{rgb}{0, 0, .4}
\definecolor{grey}{rgb}{.7, .7, .7}
\usepackage[breaklinks]{hyperref}
\hypersetup{
    colorlinks=true,
    linkcolor=darkblue,
    anchorcolor=darkblue,
    citecolor=darkblue,
    pagecolor=darkblue,
    urlcolor=darkblue,
    pdftitle={},
    pdfauthor={}
}

\newtheorem{theorem}{Theorem}[section]
\newtheorem{lemma}[theorem]{Lemma}

\theoremstyle{definition}
\newtheorem{definition}[theorem]{Definition}
\newtheorem{example}[theorem]{Example}

\numberwithin{equation}{section}

\theoremstyle{theorem}
\newtheorem{corollary}[theorem]{Corollary}

\newtheorem{proposition}[theorem]{Proposition}

\newtheorem{identity}[equation]{Identity}
\newtheorem{interpretation}{Interpretation}

\newcommand{\ul}[1]{\underline{#1}}

\newcommand{\hs}{}  % space
\newcommand{\hf}{\bullet}  % 1-entry
\usepackage[all, dvips, knot]{xy}
\usepackage[all, knot, dvips, color]{xypic}
\xyoption{arc}

\newcommand{\w}{\mathsf{w}}

\newcommand{\tld}{\widetilde}

\newcommand{\subsets}[2]{\genfrac{[}{]}{0pt}{}{#1}{#2}}
\newcommand{\qbinom}[2]{\genfrac{[}{]}{0pt}{}{#1}{#2}_q}

\renewcommand{\L}{\ensuremath{\textup{\sc l}}}
\newcommand{\M}{\ensuremath{\textup{\sc m}}}
\newcommand{\R}{\ensuremath{\textup{\sc r}}}
\newcommand{\inv}{\ensuremath{\textup{inv}}}
\newcommand{\calR}{\ensuremath{\mathcal{R}}}

\renewcommand{\l}{\ell}

\newcommand{\abacus}[1]{{ \tiny \xymatrix @-2.3pc { #1 } }}
\newcommand{\ci}[1]{{\xy*{ #1 }*\cir<7pt>{}\endxy}}  % circled entry
\newcommand{\nc}[1]{{\xy*{ #1 }*i\cir<7pt>{}\endxy}}  % entry that is not circled

\begin{document}

\title{The enumeration of fully commutative affine permutations}

\begin{abstract}
We give a generating function for the fully commutative affine
permutations enumerated by rank and Coxeter length, extending formulas due to
Stembridge and Barcucci--Del Lungo--Pergola--Pinzani.  For fixed rank, the
length generating functions have coefficients that are periodic with period
dividing the rank.  In the course of proving these formulas, we obtain results
that elucidate the structure of the fully commutative affine permutations.
\end{abstract}

\author{Christopher R.\ H.\ Hanusa}
\address{Department of Mathematics \\ Queens College (CUNY) \\ 65-30 Kissena Blvd. \\ Flushing, NY 11367}
\email{\href{mailto:chanusa@qc.cuny.edu}{\texttt{chanusa@qc.cuny.edu}}}
\urladdr{\url{http://qc.edu/\~chanusa/}}

\author{Brant C. Jones}
\address{Department of Mathematics \\ One Shields Avenue \\ University of California \\  Davis, CA 95616}
\email{\href{mailto:brant@math.ucdavis.edu}{\texttt{brant@math.ucdavis.edu}}}
\urladdr{\url{http://www.math.ucdavis.edu/\~brant/}}

\thanks{The second author received support from NSF grant DMS-0636297.}

%\dedicatory{}


%\date{\today}

\maketitle

%%%%%%%%%%%%%%%%%%%%%%%%%%%%%%%%%%%%%%%%%%%%%%%%%%%%%%%%%%%%%%%%%%%%%
%  Begin content
%%%%%%%%%%%%%%%%%%%%%%%%%%%%%%%%%%%%%%%%%%%%%%%%%%%%%%%%%%%%%%%%%%%%%

%%%%%%%%%%%%%%%%%%%%%%%%%%%%%%%%%%%%%%%%%%%%%%%%%%%%%%%%%%%%%%%%%%%%%
%  Section
%%%%%%%%%%%%%%%%%%%%%%%%%%%%%%%%%%%%%%%%%%%%%%%%%%%%%%%%%%%%%%%%%%%%%
%\tableofcontents
\section{Introduction}

Let $W$ be a Coxeter group.  An element $w$ of $W$ is \em fully commutative \em
if any reduced expression for $w$ can be obtained from any other using only
commutation relations among the generators.  For example, if $W$ is simply laced
then the fully commutative elements of $W$ are those with no $s_i s_j s_i$
factor in any reduced expression, where $s_i$ and $s_j$ are any noncommuting
generators.

The fully commutative elements form an interesting class of Coxeter group
elements with many special properties relating to smoothness of Schubert
varieties \cite{fan1}, Kazhdan--Lusztig polynomials and $\mu$-coefficients
\cite{b-w,green-tr}, Lusztig's $a(w)$-function \cite{bremke-fan,shi-a}, and decompositions
into cells \cite{shi-cells,fan-stembridge,green-losonczy-cells}.  Some of the
properties carry over to the freely braided and maximally clustered elements
introduced in \cite{g-l1,g-l2,losonczy-mc}.  At the level of the Coxeter group,
\cite{s1} shows that each fully commutative element $w$ has a unique labeled
partial order called the \em heap \em of $w$ whose linear extensions encode all
of the reduced expressions for $w$.  

Stembridge \cite{s1} (see also \cite{fan-thesis,graham-thesis}) classified the Coxeter groups having
finitely many fully commutative elements.
In \cite{s2}, he then enumerated the total number of fully commutative
elements in each of these Coxeter groups.  The type $A$ series yields the
Catalan numbers, a result previously given in \cite{b-j-s}.  Barcucci et al.\
\cite{barcucci} have enumerated the fully commutative permutations by Coxeter
length, obtaining a $q$-analogue of the Catalan numbers.  Our main result in
Theorem~\ref{t:all} is an analogue of this result for the affine symmetric
group.

In a general Coxeter group, the fully commutative elements index a basis for a
quotient of the corresponding Hecke algebra
\cite{graham-thesis,fan-thesis,green-losonczy-canonical}.  In type $A$, this
quotient is the Temperley--Lieb algebra; see \cite{temperley-lieb,westbury}.
Therefore, our result can be interpreted as a graded dimension formula for the
affine analogue of this algebra.

In Section~\ref{s:background}, we introduce the necessary definitions and
background information.  In Section~\ref{s:enumeration}, we enumerate the fully
commutative affine permutations by decomposing them into several subsets.  The
formula that we obtain turns out to involve a ratio of $q$-Bessel functions as
described in \cite{barcucci_q_bessel} arising as the solution obtained by
\cite{MBMcolumnconvex} of a certain recurrence relation on the generating
function.  A similar phenomenon occurred in \cite{barcucci}, and our work can be
viewed as a description of how to lift this formula to the affine case.  It
turns out that the only additional ingredients that we need for our formula are
certain sums and products of $q$-binomial coefficients.  In
Section~\ref{s:numerical}, we prove that for fixed rank, the coefficients of
the length generating functions are periodic with period dividing the rank.
This result gives another way to determine the generating functions by computing
the finite initial sequence of coefficients until the periodicity takes over.
We mention some further questions in Section~\ref{s:future}.

%%%%%%%%%%%%%%%%%%%%%%%%%%%%%%%%%%%%%%%%%%%%%%%
%   New Section
%%%%%%%%%%%%%%%%%%%%%%%%%%%%%%%%%%%%%%%%%%%%%%%

\section{Background}
\label{s:background}

In this section, we introduce the affine symmetric group, abacus diagrams for
minimal length coset representatives, and $q$-binomial coefficients.

\subsection{The affine symmetric group}\label{s:affsym}

We view the symmetric group $S_n$ as the Coxeter group of type $A$ with
generating set $S = \{ s_1, \ldots, s_{n-1} \}$ and relations of the form $(s_i
s_{i \pm 1})^3 = 1$ together with $(s_{i}s_{j})^{2} = 1$ for $| i - j | \geq 2$
and ${s_i}^2 = 1$.  We denote $\bigcup_{n \geq 1} S_n$ by $S_{\infty}$ and call
$n(w)$ the minimal rank $n$ of $w \in S_n \subset S_{\infty}$.  The affine symmetric
group $\widetilde{S}_n$ is also a Coxeter group; it is generated by
$\widetilde{S} = S \cup \{ s_0 \}$ with the same relations as in the symmetric
group together with ${s_0}^2 = 1$, $(s_{n-1} s_0)^3 = 1$, $(s_0 s_1)^3 = 1$,
and $(s_0 s_j)^2 = 1$ for $2 \leq j \leq n-2$.

Recall that the products of generators from $S$ or $\widetilde{S}$ with a
minimal number of factors are called \em reduced expressions\em, and $\l(w)$ is
the length of such an expression for an (affine) permutation $w$.  Given an
(affine) permutation $w$, we represent reduced expressions for $w$ in sans
serif font, say $\w=\w_{1} \w_{2}\cdots \w_{p}$ where each $\w_{i} \in S$ or
$\widetilde{S}$.  We call any expression of the form $s_i s_{i \pm 1} s_i$ a
\em short braid\em, where the indices $i, i \pm 1$ are taken mod $n$ if we are
working in $\widetilde{S}_n$.  There is a well-known theorem of Matsumoto
\cite{matsumoto} and Tits \cite{t}, which states that any reduced expression
for $w$ can be transformed into any other by applying a sequence of relations
of the form $(s_i s_{i \pm 1})^3 = 1$ (where again $i, i \pm 1$ are taken mod
$n$ in $\widetilde{S}_n$) together with $(s_{i}s_{j})^{2} = 1$ for $| i - j | >
1$.  We say that $s_i$ is a \em left descent \em for $w \in \widetilde{S}_n$ if
$\l(s_i w) < \l(w)$ and we say that $s_i$ is a \em right descent \em for $w \in
\widetilde{S}_n$ if $\l(w s_i) < \l(w)$.

As in \cite{s1}, we define an equivalence relation on the set of reduced
expressions for an (affine) permutation by saying that two reduced expressions are in the
same \em commutativity class \em if one can be obtained from the other by a
sequence of \em commuting moves \em of the form $s_i s_j \mapsto s_j s_i$ where
$|i-j| \geq 2$.  If the reduced expressions for a permutation $w$ form a single
commutativity class, then we say $w$ is \em fully commutative\em.

If $\w = \w_1 \cdots \w_k$ is a reduced expression for any permutation, then
following \cite{s1} we define a partial ordering on the indices $\{1, \ldots,
k\}$ by the transitive closure of the relation $i \lessdot j$ if $i < j$ and
$\w_i$ does not commute with $\w_j$.  We label each element $i$ of the poset by
the corresponding generator $\w_i$.  It follows quickly from the definition that
if $\w$ and $\w'$ are two reduced expressions for an element $w$ that are in the
same commutativity class then the labeled posets of $\w$ and $\w'$ are
isomorphic.  This isomorphism class of labeled posets is called the \em heap \em
of $\w$, where $\w$ is a reduced expression representative for a commutativity
class of $w$.  In particular, if $w \in S_n$ is fully commutative then it has a
single commutativity class, and so there is a unique heap of $w$.  Cartier and
Foata \cite{cartier-foata} were among the first to study heaps of dimers, which
were generalized to other settings by Viennot \cite{viennot}.

We also refer to elements in the symmetric group by the \textit{one-line
notation} $w=[w_{1}w_{2} \cdots w_{n}]$, where $w$ is the bijection mapping $i$
to $w_{i}$.  Then the generators $s_{i}$ are the adjacent transpositions
interchanging the entries $i$ and $i+1$ in the one-line notation.  Let
$w=[w_1 \cdots w_n]$, and suppose that $p=[p_1 \cdots p_k]$ is another permutation in $S_{k}$
for $k\leq n$.  We say $w$ \textit{contains the permutation pattern} $p$ or $w$
\textit{contains} $p$ \textit{as a one-line pattern} whenever there exists a
subsequence $1\leq i_{1}<i_{2}<\cdots<i_{k}\leq n$ such that
\begin{equation*}
w_{i_{a}} < w_{i_{b}} \text{ if and only if } p_{a} < p_{b}
\end{equation*}
for all $1 \leq a < b \leq k$.  We call $(i_1, i_2, \ldots, i_k)$ the \em
pattern instance\em.  For example, $[\ul{5}32\ul{4}\ul{1}]$ contains the
pattern $[321]$ in several ways, including the underlined subsequence.  If $w$
does not contain the pattern $p$, we say that $w$ \em avoids \em $p$.

The \em affine symmetric group \em $\widetilde{S}_{n}$ is realized in
\cite[Chapter 8]{b-b} as the group of bijections $w: \mathbb{Z} \rightarrow
\mathbb{Z}$ satisfying $w(i + n) = w(i) + n$ and $\sum_{i = 1}^n w(i) =
\sum_{i=1}^n i = {{n+1} \choose 2}$.  We call the infinite sequence 
\[ (\ldots, w(-1), w(0), w(1), w(2), \ldots, w(n), w(n+1), \ldots) \]
the \em complete notation \em for $w$ and 
\[ [w(1), w(2), \ldots, w(n)] \]
the \em base window \em for $w$.  By definition, the entries of the base window
determine $w$ and its complete notation.  Moreover, the entries of the base
window can be any set of integers that are normalized to sum to ${ {n+1} \choose
2}$ and such that the entries form a permutation of the residue classes in
$\mathbb{Z}/(n \mathbb{Z})$ when reduced mod $n$.  That is, no two entries of
the base window have the same residue mod $n$.  With these considerations in
mind, we will represent an affine permutation using an abacus diagram together
with a finite permutation.

To describe this, observe that $S_n$ acts on the base window by permuting the entries, which
induces an action of $S_n$ on $\mathbb{Z}$.  In this action, the Coxeter
generator $s_i$ simultaneously interchanges $w(i + kn)$ with $w(i+1+kn)$ for all
$k \in \mathbb{Z}$.  Moreover, the affine generator $s_0$ interchanges all $w(k
n)$ with $w(kn + 1)$.  Hence, $S_{n}$ is a parabolic subgroup of
$\widetilde{S}_{n}$.  We form the parabolic quotient 
\[ \widetilde{S}_{n} / S_{n} = \{ w \in \widetilde{S}_{n} : \l(w s_i) > \l(w) \text{ for all $s_i$ where $1 \leq i \leq n-1$} \}. \]
By a standard result in the theory of Coxeter groups, this set gives a unique
representative of minimal length from each coset $w S_{n}$ of
$\widetilde{S}_{n}$.  For more on this construction, see \cite[Section
2.4]{b-b}.  In our case, the base window of the minimal length coset
representative of an element is obtained by ordering the entries that appear in
the base window increasingly.  This construction implies that, as sets, the
affine symmetric group can be identified with the set $(\widetilde{S}_n / S_n)
\times S_n$.  The minimal length coset representative determines which entries
appear in the base window, and the finite permutation orders these entries in
the base window.

We say that $w$ has a \em descent \em at $i$ whenever $w(i) > w(i+1)$.  Note
that if $w$ has a descent at $i$, then $s_{(\textup{$i$ mod $n$})}$ is a right descent in
the usual Coxeter theoretic sense that $\l(w s_i) < \l(w)$.

%%%%%%%%%%%%%%%%%%%%%%%%%%%%%%%%%%%%%%%%%%%%%%%
%   New Subsection
%%%%%%%%%%%%%%%%%%%%%%%%%%%%%%%%%%%%%%%%%%%%%%%

\subsection{Abacus diagrams}\label{s:abacus}

The abacus diagrams of \cite{jk} give a combinatorial model for the minimal
length coset representatives in $\widetilde{S}_n / S_n$.  Other combinatorial
models and references for these are given in \cite{BJV}.

An \em abacus diagram \em is a diagram containing $n$ columns
labeled $1, 2, \ldots, n$, called \em runners\em.  The horizontal rows are called
\em levels \em and runner $i$ contains entries labeled by $r n + i$ on each
level $r$ where $-\infty < r < \infty$.  We draw the abacus so that each runner
is vertical, oriented with $-\infty$ at the top and $\infty$ at the bottom, with
runner $1$ in the leftmost position, increasing to runner $n$ in the rightmost
position.  Entries in the abacus diagram may be circled; such circled elements
are called \em beads\em.  Entries that are not circled are called \em gaps\em.
The linear ordering of the entries given by the labels $r n + i$ is called the
\em reading order \em of the abacus which corresponds to scanning left to right,
top to bottom.

We associate an abacus to each minimal length coset representative $w \in
\widetilde{S}_n / S_n$ by drawing beads down to level $w_i$ in runner $i$ for
each $1 \leq i \leq n$ where $\{ w_1, w_2, \ldots, w_n \}$ is the set of
integers in the base window of $w$, with no two having the same residue mod $n$.
Since the entries $w_i$ sum to ${ {n+1} \choose 2}$, we call the abacus
constructed in this way \em balanced\em.  It follows from the construction that
the Coxeter length of the minimal length coset representative can be determined
from the abacus.

\begin{proposition}\label{p:length}
Let $w \in \widetilde{S}_n / S_n$ and form the abacus for $w$ as described
above.  Let $m_i$ denote the number of gaps preceding the lowest bead of runner
$i$ in reading order, for each $1 \leq i \leq n$.  Then, the Coxeter length
$\l(w)$ is $\sum_{i = 1}^n m_i$.
\end{proposition}
\begin{proof}
This result is part of the folklore of the subject.  One proof can be obtained
by combining Propositions 3.2.5 and 3.2.8 of \cite{BJV}. 
\end{proof}

\begin{example}
The affine permutation $\widetilde{w} = [-1, -4, 14, 1]$ is identified with the
pair $(w^0, w)$ where $w$ is the finite permutation $s_1 s_3 = [2143]$ which
sorts the elements of the minimal length coset representative $w^0 = [-4, -1, 1,
14]$.  Note that the entries of $w^0$ sum to ${ 5 \choose 2 } = 10$.  The abacus
of $w^0$ is shown below.
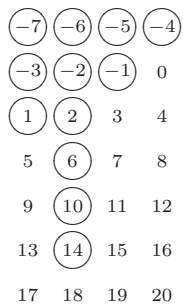
\[ 
\abacus {
\ci{-7} &  \ci{-6} & \ci{-5} & \ci{-4} \\
 \ci{-3} &  \ci{-2} & \ci{-1} & \nc{0} \\
 \ci{1} &  \ci{2} & \nc{3} & \nc{4} \\
 \nc{5} & \ci{6} & \nc{7} & \nc{8} \\
 \nc{9} & \ci{10} & \nc{11} & \nc{12} \\
 \nc{13} & \ci{14} & \nc{15} & \nc{16} \\
 \nc{17} & \nc{18} & \nc{19} & \nc{20} \\
} 
\]
From the abacus, we see that $w^0$ has Coxeter length $1+10+0+0 = 11$.  For
example, the ten gaps preceding the lowest bead in runner 2 are $13, 12, 11, 9, 8,
7, 5, 4, 3$, and $0$.  Hence, $\widetilde{w}$ has length $\l(w^0) + \l(w) = 13$.
\end{example}

In this work, we are primarily concerned with the fully commutative affine
permutations.  Green has given a criterion for these in terms of the complete
notation for $w$.  His result is a generalization of a theorem from
\cite{b-j-s} which states that $w \in S_n$ is fully commutative if and only if
$w$ avoids $[321]$ as a permutation pattern.  

\begin{theorem}{\cite{green-321}}\label{t:321-av}
Let $w \in \tld{S}_n$.  Then, $w$ is fully commutative if and only if 
there do not exist integers $i < j < k$ such that $w(i) > w(j) > w(k)$.
\end{theorem}

Observe that even though the entries in the base window of a minimal length
coset representative are sorted, the element may not be fully commutative
by Theorem~\ref{t:321-av}.  For example, if we write the element $w^0 = [-4, -1,
1, 14]$ in complete notation
\[ w^0 = (\ldots, -8, -5, -3, {\bf 10}, -4, -1, {\bf 1}, 14, {\bf 0}, 3, 5, 18, \ldots) \]
we obtain a $[321]$-instance as indicated in boldface.

In order to more easily exploit this phenomenon, we slightly modify the
construction of the abacus.  Observe that the length formula in
Proposition~\ref{p:length} depends only on the relative positions of the beads
in the abacus, and is unchanged if we shift every bead in the abacus exactly
$k$ positions to the right in reading order.  Moreover, each time we shift the
beads one unit to the right, we change the sum of the entries occurring on the
lowest bead in each runner by exactly $n$.  In fact, this shifting corresponds
to shifting the base window inside the complete notation.  Therefore, we may
define an abacus in which all of the beads are shifted so that position $n+1$
becomes the first gap in reading order.  We call such abaci \em normalized\em.
Although the entries of the lowest beads in each runner will no longer sum to
${ {n+1} \choose 2}$, we can reverse the shifting to recover the balanced
abacus.  Hence, this process is a bijection on abaci, and we may assume from
now on that our abaci are normalized.

\begin{proposition}\label{p:long_short}
Let $A$ be a normalized abacus for $w^0 \in \widetilde{S}_n / S_n$, and suppose
the last bead occurs at entry $i$.  Then, $w^0$ is fully commutative if and only
if the lowest beads on runners of $A$ occur only in positions that are a subset
of $\{1, 2, \ldots, n\} \cup \{ i-n+1, i-n+2, \ldots, i \}$.
\end{proposition}
\begin{proof}
By construction, position $n+1$ is the first gap in $A$, so the lowest bead on
runner $1$ occurs at position $1$, and all of the positions $2, 3, \ldots, n$
are occupied by beads.  Suppose there exists a lowest bead at position $j$ with
$n < j < i - n + 1$, and consider the complete notation for $w^0$ obtained by
arranging the positions of the lowest beads in each runner sequentially in the
base window.  We obtain a $[321]$-instance in positions $i-n$ from the window
immediately preceding the base window, $j$ from the base window, and $n+1$ from
the window immediately succeeding the base window.  Hence, $w^0$ is not fully
commutative by Theorem~\ref{t:321-av}.

Otherwise, there does not exist $j$ such that $n < j < i - n + 1$.  Hence, each
of the entries in the base window belongs to $\{1, 2, \ldots, n\}$ in which case
we say that the entry is \em short\em, or it belongs to $\{ i-n+1, i-n+2,
\ldots, i\}$ in which case we say that the entry is \em long\em.  If $i \geq
2n$, then these designations are disjoint; otherwise, some entries may be both
short and long.  For example, in the first normalized abacus shown
Figure~\ref{fig:runners} below, the entries $1$, $3$, and $4$ are short while entries
$12$ and $10$ are long.  The corresponding complete notation is $(\ldots, 7
| 1, 3, 4, 10, 12 | 6, 8, 9, 15, 17 | 11, \ldots)$.

Because $w^0$ is constructed with an increasing base window, any entry $w^0(b)$
that is equivalent mod $n$ to one of the short entries has the property that
$w^0(c) > w^0(b)$ for all $c > b$.  Therefore, the only inversions $w^0(a) >
w^0(b)$ for $a < b$ in the complete notation occur between an entry $w^0(a)$
that is equivalent mod $n$ to a long entry with an entry $w^0(b)$ that is
equivalent mod $n$ to a short entry.  Hence, $w^0$ is $[321]$-avoiding, from which it follows that $w^0$ is fully
commutative by Theorem~\ref{t:321-av}.
\end{proof}

We distinguish between two types of fully commutative elements through the position of the last bead in its normalized abacus $A$.  If the last bead occurs in a position $i>2n$, then we call the element a {\em long} element.  Otherwise, the last bead occurs in a position $n\leq i\leq 2n$, and we call the element a {\em short} element.  As evidenced in Section \ref{s:enumeration}, the long fully commutative elements have a nice structure that allows for an elegant enumeration; the short elements lack this structure.

%%%%%%%%%%%%%%%%%%%%%%%%%%%%%%%%%%%%%%%%%%%%%%%
%   New Subsection
%%%%%%%%%%%%%%%%%%%%%%%%%%%%%%%%%%%%%%%%%%%%%%%

\subsection{$q$-analogs of binomial coefficients}
Calculations involving $q$-analogs of combinatorial objects often involve
$q$-analogs of counting functions.  A few standard references on the subject
are \cite{Andrews,GouldenJackson,ec1}.  Define
$(a,q)_n=(1-a)(1-aq)\cdots(1-aq^{n-1})$ and $(q)_n=(q,q)_n$.  The {\em
$q$-binomial coefficient} $\qbinom{n}{k}$ (also called the {\em Gaussian
polynomial}) is a $q$-analog of the binomial coefficient $\binom{n}{k}$.  To
calculate a $q$-binomial coefficient directly, we use the formula
\begin{equation}\label{e:q_binomial}
\qbinom{n}{k}=\frac{(1-q^n)(1-q^{n-1})\cdots(1-q^{n-k+1})}{(1-q^k)(1-q^{k-1})\cdots(1-q^{1})}=\frac{(q)_n}{(q)_k(q)_{n-k}}.
\end{equation}
Just as with ordinary binomial coefficients, $q$-binomial coefficients have multiple combinatorial interpretations and satisfy many identities, a few of which are highlighted below.
\begin{interpretation}{\cite[Proposition 1.3.17]{ec1}}
Let $M$ be the multiset $M=\{1^k,2^{n-k}\}$. For an ordering $\pi$ of the $n$ elements of $M$, the number of inversions of $\pi$, denoted $\inv(\pi)$, is the number of instances of two entries $i$ and $j$ such that $i<j$ and $\pi(i)>\pi(j)$.  Then $\qbinom{n}{k}=\sum_{\pi} q^{\inv(\pi)}$.
\label{interp:intertwining}
\end{interpretation}
\begin{interpretation}{\cite[Proposition 1.3.19]{ec1}}
Let $\Lambda$ be the set of partitions $\lambda$ whose Ferrers diagram fits inside a $k\times(n-k)$ rectangle.  Then $\qbinom{n}{k}=\sum_{\lambda\in\Lambda}q^{|\lambda|}$, where $|\lambda|$ denotes the number of boxes in the diagram of $\lambda$.
\label{interp:Ferrers}
\end{interpretation}
\begin{interpretation}
Let $\subsets{n}{k}$ be the set of subsets of $[n] = \{1, 2, \ldots, n\}$ of
size $k$. Given $\calR=\{r_1,\ldots,r_k\}\in \subsets{n}{k}$, define
$|\calR|=\sum_{j=1}^k (r_j-j)$. Then $\qbinom{n}{k}=\sum_{\calR\in\subsets{n}{k}}q^{|\calR|}$.
\label{interp:subsets}
\end{interpretation}
\begin{proof}
There is a standard bijection between the diagram of a partition $\lambda$ drawn
in English notation inside a $(n-k) \times k$ rectangle and lattice paths of
length $n$ consisting of down and left steps that contain $k$ left steps.  This
bijection is given by tracing the lattice path formed by the boundary of the
partition $\lambda$ from the upper right to the lower left corners of the
bounding rectangle.  We can obtain another bijection to subsets
$\calR=\{r_1,\ldots,r_k\}\in \subsets{n}{k}$ by recording the index $r_j \in \{1,
2, \ldots, n\}$ of the horizontal steps of the path in $\calR$ for each $j = 1,
\ldots, k$.  Then, the number of boxes of $\lambda$ that are added in the column
above each horizontal step is precisely $(r_j - j)$.  Hence,
Interpretation~\ref{interp:subsets} follows from
transposing the Ferrers diagrams in Interpretation~\ref{interp:Ferrers}.
\end{proof}

Interpretation~\ref{interp:intertwining} is used most frequently in this
article.  Interpretation~\ref{interp:subsets} is used in Section~\ref{s:long}
when counting long fully commutative elements.
%Interpretation~\ref{interp:Ferrers} can be used to construct bijections between
%the interpretations.

The following identities follow directly from Equation~\ref{e:q_binomial}.

\begin{identity}
$\displaystyle \qbinom{n}{k}=\qbinom{n}{n-k}.$
\label{i:n-k}
\end{identity}
\begin{identity}
$\displaystyle (1-q^{n-k})\qbinom{n}{k}=(1-q^{n})\qbinom{n-1}{k}.$
\label{i:convert}
\end{identity}

%%%%%%%%%%%%%%%%%%%%%%%%%%%%%%%%%%%%%%%%%%%%%%%
%   New Section
%%%%%%%%%%%%%%%%%%%%%%%%%%%%%%%%%%%%%%%%%%%%%%%

\section{Decomposition and enumeration of fully commutative elements}
\label{s:enumeration}

Let $S_n^{FC}$ denote the set of fully commutative permutations in $S_n$.  In
the following result, Barcucci et al. enumerate these elements by Coxeter
length.

\begin{theorem}{\cite{barcucci}} Let $C(x,q)=\sum_{n \geq 0} \sum_{w \in
    S_n^{FC}} x^n q^{\l(w)}$.  Then,
\[ C(x,q) = \frac{\sum_{n\geq 0}(-1)^nx^{n+1}q^{(n(n+3))/2}/(x,q)_{n+1}(q,q)_n}
{\sum_{n\geq 0}(-1)^nx^{n}q^{(n(n+1))/2}/(x,q)_{n}(q,q)_n} \]
\label{t:barcucci}
\end{theorem}

This formula is a ratio of $q$-Bessel functions as described in
\cite{barcucci_q_bessel}.  It arises as the solution obtained by
\cite{MBMcolumnconvex} of a recurrence relation given on the generating
function.  We will encounter such a recurrence in the proof of
Lemma~\ref{l:S_2} below.

We enumerate the fully commutative elements $\tld{w}\in \tld{S}_n$ by
identifying each as the product of its minimal length coset representative
$w^0\in \widetilde{S}_n / S_n$ and a finite permutation $w\in S_n$ as described
in Section~\ref{s:abacus}.  Recall that we decompose the set of fully
commutative elements into long and short elements.  The elements with a short
abacus structure break down into those where certain entries intertwine and
those in which there is no intertwining.  When we assemble these cases, we
obtain our main theorem.

\begin{theorem}\label{t:main}
Let $\tld{S}_n^{FC}$ denote the set of fully commutative affine permutations in
$\tld{S}_n$, and let $G(x,q)=\sum_{n \geq 0} \sum_{w \in \tld{S}_n^{FC}} x^n
q^{\l(w)}$, where $\l(w)$ denotes the Coxeter length of $w$.  Then, 
\[ G(x,q) = \Bigg(\sum_{n \geq 0} \frac{x^nq^n}{1-q^n} \sum_{k=1}^{n-1}\qbinom{n}{k}^{\,2}\Bigg) + C(x,q)+  \Bigg(\sum_{R,L\geq 1} q^{R+L-1} \qbinom{L+R-2}{L-1} S(x,q)\Bigg), \]
where $C(x,q)$ is given by Theorem~\ref{t:barcucci}, and the component parts of
$S(x,q)=S_I(x,q)+S_0(x,q)+S_1(x,q)+S_2(x,q)$ are given in Lemmas~\ref{l:S_I},
\ref{l:S_0}, \ref{l:S_1}, and \ref{l:S_2}, respectively.
\label{t:all}
\end{theorem}
The first summand of $G(x,q)$ counts the long elements, while the remaining
summands count the short elements.  This theorem will be proved in
Section~\ref{s:main_proof} below.

%%%%%%%%%%%%%%%%%%%%%%%%%%%%%%%%%%%%%%%%%%%%%%%
%   New Subsection
%%%%%%%%%%%%%%%%%%%%%%%%%%%%%%%%%%%%%%%%%%%%%%%

\subsection{Long elements}
\label{s:long}

In this section, we enumerate the long elements.  Recall that the last bead in the normalized abacus for these elements occurs in position $> 2n$.

\begin{lemma}
For fixed $n \geq 0$, we have
\[ \sum_{w \in \tld{S}_n^{FC} \text{ such that $w$ is long }} q^{\l(w)} = \frac{q^n}{1-q^n}\sum_{k=1}^{n-1}\qbinom{n}{k}^{\,2}. \]
\label{l:long}
\end{lemma}

\begin{proof}
Fix a long fully commutative element, and define the set of {\em long runners}
$\calR$ of its normalized abacus $A$ to be the set of runners
$\{r_1,\ldots,r_k\}\subset [n]\setminus\{1\}$ in which there exists a bead in
position $n+r_j$ for $1\leq j\leq k$.   We will enumerate the long fully
commutative elements by conditioning on $k=|\calR|$, the size of the set of long
runners of its normalized abacus.  Note that by Proposition~\ref{p:long_short},
all subsets $\calR\subset[n]\setminus\{1\}$ are indeed the set of long runners
for some fully commutative element.

For a fixed $\calR$, there is an infinite family of abaci
$\{A_i^{\calR}\}_{i\geq 1}$, each having beads in positions $n+r_j$ for $r_j\in
\calR$, together with $i$ additional beads that are placed sequentially in the
long runners in positions larger than $2n$.  See Figure~\ref{fig:runners} for
an example.

\begin{figure}
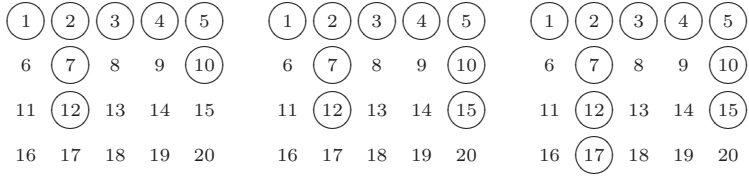

\[ 
\abacus {
 \ci{1} &  \ci{2} & \ci{3} & \ci{4} & \ci{5} \\
 \nc{6} & \ci{7} & \nc{8} & \nc{9} & \ci{10} \\
 \nc{11} &  \ci{12} & \nc{13} & \nc{14} & \nc{15} \\
 \nc{16} & \nc{17} & \nc{18} & \nc{19} & \nc{20} \\
} 
\quad
\abacus {
 \ci{1} &  \ci{2} & \ci{3} & \ci{4} & \ci{5} \\
 \nc{6} & \ci{7} & \nc{8} & \nc{9} & \ci{10} \\
 \nc{11} &  \ci{12} & \nc{13} & \nc{14} & \ci{15} \\
 \nc{16} & \nc{17} & \nc{18} & \nc{19} & \nc{20} \\
} 
\quad
\abacus {
 \ci{1} &  \ci{2} & \ci{3} & \ci{4} & \ci{5} \\
 \nc{6} & \ci{7} & \nc{8} & \nc{9} & \ci{10} \\
 \nc{11} &  \ci{12} & \nc{13} & \nc{14} & \ci{15} \\
 \nc{16} & \ci{17} & \nc{18} & \nc{19} & \nc{20} \\
} 
\]
\caption{The first three members---$A_1^{\calR}$, $A_2^{\calR}$, and $A_3^{\calR}$---of the infinite family of abaci $\{A_i^{\calR}\}$ with long runners $\calR=\{2,5\}$.  These abaci correspond to certain long fully commutative elements of $\tld{S}_5$.}
\label{fig:runners}
\end{figure}

By Proposition \ref{p:length}, the Coxeter length of the minimal length coset
representative $w^0\in \widetilde{S}_n / S_n$ having $A_i^{\calR}$ as its
abacus is $i(n-k) +\sum_{j=1}^k (r_j-j)$.  In addition, $w^0$ has base window
$[aa\cdots abb\cdots b]$, where the $(n-k)$ numbers $a$ are all at most $n$,
and the $k$ numbers $b$ are all at least $n+2$.  The finite permutations $w$
that can be applied to this standard window may not invert any of the larger
numbers ($b$'s) without creating a $[321]$-pattern with $n+1$ in the window
following the standard window.   Similarly, none of the $a$'s can be inverted.
All that remains is to intersperse the $a$'s and the $b$'s, keeping track of
how many transpositions are used.  This contributes exactly $\qbinom{n}{k}$ to
the Coxeter length, by Interpretation~\ref{interp:intertwining} of
$\qbinom{n}{k}$.

Therefore, the generating function for the long fully commutative elements of $\tld{S}_n$ by Coxeter length is
$$\sum_{k=1}^{n-1}\qbinom{n}{k}\sum_{i\geq 1}\sum_{\substack{\calR\subset[n]\setminus\{1\} \\ |\calR|=k}} q^{i(n-k) +\sum_{j=1}^k (r_j-j)}.$$
Taking the sum over $i$ and incorporating a factor of $q^{-k}$ into the summation in the exponent of $q$ yields
$$\sum_{k=1}^{n-1}\qbinom{n}{k}\frac{q^n}{1-q^{n-k}}\sum_{\substack{\calR\subset[n]\setminus\{1\} \\ |\calR|=k}} q^{\sum_{j=1}^k (r_j-1-j)}.$$
Reindexing the entries $r_j$ to be from $1$ to $n-1$ instead of from $2$ to $n$ gives
$$\sum_{k=1}^{n-1}\qbinom{n}{k}\frac{q^n}{1-q^{n-k}}\sum_{\substack{\calR\subset[n-1] \\ |\calR|=k}} q^{\sum_{j=1}^k (r_j-j)},$$
which simplifies by Interpretation~\ref{interp:subsets} of the $q$-binomial coefficients to
$$\sum_{k=1}^{n-1}\qbinom{n}{k}\qbinom{n-1}{k}\frac{q^n}{1-q^{n-k}}.$$
Applying Identity~\ref{i:convert} proves the desired result.
\end{proof}

%%%%%%%%%%%%%%%%%%%%%%%%%%%%%%%%%%%%%%%%%%%%%%%
%   New Subsection
%%%%%%%%%%%%%%%%%%%%%%%%%%%%%%%%%%%%%%%%%%%%%%%

\subsection{Short elements}

The normalized abacus of every short fully commutative element has a particular
structure.  There must be a gap in position $n+1$, and for runners $2\leq i
\leq n$, the lowest bead is either in position $i$ or $n+i$.  In the following
arguments, we will assign a status to each runner, depending on the position of
the lowest bead in that runner.

\begin{definition}\label{d:lmr}
An \em $\R$-entry \em is a bead lying in some position $> n$.  Let $n+j$ be the
position of the last $\R$-entry, or set $j = n$ if there are no $\R$-entries; an
\em $\M$-entry \em is a lowest bead lying in position $i$ where $j+1 \leq i \leq
n$.  Note that it is possible that there do not exist any $\M$-entries.  The \em
$\L$-entries \em are the remaining lowest beads in position $i$ for $i\leq j$.
This assigns a status \em left\em, \em middle\em, or \em right \em to each entry
of the base window, depending on the position of the lowest bead in the
corresponding runner.  We will call an abacus containing $L$ $\L$-entries, $M$
$\M$-entries, and $R$ $\R$-entries an {\em $(L)(M)(R)$ abacus}. 
\end{definition}

\begin{example}
Figure~\ref{fig:312abaci} shows the three $(3)(1)(2)$ abaci.  In each case, $6$
is the unique $\M$-entry and $11$ is an $\R$-entry.  In the first abacus, the
$\L$-entries are $\{1, 3, 4\}$ and the $\R$-entries are $\{8, 11\}$.
\end{example}

The rationale for this assignment is that in the base window of a fully
commutative element, not of type $(n)(0)(0)$, neither the $\L$-entries nor the
$\R$-entries can have a descent amongst themselves, respectively.  To see
this, consider the contrary where two $\R$-entries have a descent.  These two
entries, along with the $n+1$ entry in the window following the standard
window, form a $[321]$-instance.  Similarly, the last $\R$-entry in the window
previous to the standard window together with two $\L$-entries that have a
descent in the standard window would form a $[321]$-instance.

When the normalized abacus of a short fully commutative element has no
$\R$-entries (and therefore no $\M$-entries), the base window for its minimal
length coset representative is $[12\cdots n]$.  That is, the fully commutative
elements of $\tld{S}_n$ having this abacus are in one-to-one correspondence with
fully commutative elements of finite $S_n$.  These elements have been enumerated
in Theorem~\ref{t:barcucci}. 

\medskip
From now on, we only concern ourselves with $(L)(M)(R)$ abaci where $R>0$.
Proposition~\ref{p:lmr} proves that it is solely the parameters $L$, $M$, and
$R$ that determine the set of finite permutations that we can apply to the
minimal length coset representative, and not the exact abacus.  In
Proposition~\ref{p:fixedLMR} we determine the cumulative contribution to the
Coxeter length of the minimal length coset representative from all $(L)(M)(R)$
abaci for fixed $L$, $M$, and $R$.  

\begin{proposition}
Let $w_1^0,w_2^0\in \tld{S}_n / S_n$, each corresponding to an $(L)(M)(R)$
abacus for the same $L$, $M$, and $R$ with $R>0$. For any finite permutation
$w\in S_n$, $w_1^0 w$ is fully commutative in $\tld{S}_n$ if and only if $w_2^0
w$ is fully commutative in $\tld{S}_n$.
\label{p:lmr}
\end{proposition}
\begin{proof}
For $v \in \widetilde{S}_n$ and $i \in \mathbb{Z}$, we will say that $v(i)$ has
the same left, middle, or right status as the entry $v(\textup{$i$ mod $n$})$ of
the base window, as in Definition~\ref{d:lmr}.  Observe that $(w_1^0 w)(i)$ has
the same left, middle, or right status as $(w_2^0 w)(i)$ for all $1 \leq i \leq
n$, and that the relative order of these entries is the same for $w_1^0 w$ as
for $w_2^0 w$.

Next, suppose that $w_1^0 w$ has a $[321]$-instance with two inverted
$\L$-entries or two inverted $\R$-entries.  By construction, these entries must
occur in the same window $j$.  Then any $\R$-entry from window $j-1$ yields a
$[321]$-instance in $w_2^0 w$, and such an entry exists since we are assuming
that $R>0$.  Similarly, if the $[321]$-instance has two inverted $\R$-entries
occurring in window $j$, then any $\L$-entry from window $j+1$ yields a
$[321]$-instance in $w_2^0 w$, and such an entry exists since $1$ is always an
$\L$-entry in the base window.  Hence, $w_2^0 w$ is not fully commutative by
Theorem~\ref{t:321-av}.

Next, suppose that $w_1^0 w$ has a $[321]$-instance that includes two
$\M$-entries, at least one of which lies in window $j$.  Observe that every
$\M$-entry in window $j$ is larger than every entry in window $j-1$, and
smaller than every entry in window $j+1$.  Therefore, if the $[321]$-instance
involves two $\M$-entries then the entire $[321]$-instance must occur within
window $j$, which implies that $w$ is not fully commutative.  Thus, $w_2^0 w$
is not fully commutative.

Finally, if $w_1^0 w$ has a $[321]$-instance that includes one $\R$-entry, one
$\M$-entry, and one $\L$-entry, then all three of these entries must lie in the
same window.  Hence, neither $w$ nor $w_2^0 w$ are fully commutative.

Thus, we have shown that the result holds in all cases.
\end{proof}

\begin{proposition}
Let $L$, $M$ and $R > 0$ be fixed.  Then, we have
\[ \sum_{w} q^{\l(w)} = q^{L+R-1}\qbinom{L+R-2}{L-1}. \]
where the sum on the left is over all minimal length coset representatives $w$
having an $(L)(M)(R)$ abacus.
\label{p:fixedLMR}
\end{proposition}
\begin{proof}

Every $(L)(M)(R)$ abacus contains beads in all positions through $n$ and in
position $2n-M$ as well as gaps in position $n+1$ and all positions starting
with $2n-M+1$.  Depending on the positions of the $L-1$ remaining gaps (and
$R-1$ remaining beads), the Coxeter length of the minimal length coset
representative changes as illustrated by example in Figure \ref{fig:312abaci}.

\begin{figure}
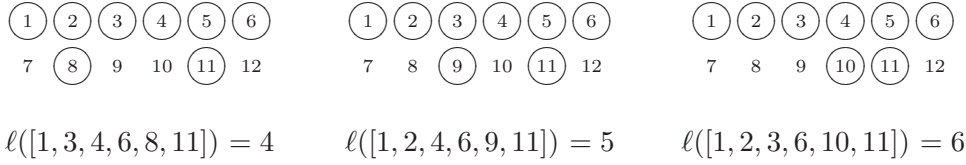

\begin{tabular}{c}
$\abacus {
 \ci{1} &  \ci{2} & \ci{3} & \ci{4} & \ci{5} & \ci{6} \\
 \nc{7} & \ci{8} & \nc{9} & \nc{10} & \ci{11} & \nc{12}\\
}$  
\\ \\
$\l([1, 3, 4, 6, 8, 11])=4$
\end{tabular}
\quad
\begin{tabular}{c}
$\abacus {
 \ci{1} &  \ci{2} & \ci{3} & \ci{4} & \ci{5} & \ci{6} \\
 \nc{7} & \nc{8} & \ci{9} & \nc{10} & \ci{11} & \nc{12}\\
}$ 
\\ \\
$\l([1, 2, 4, 6, 9, 11])=5$
\end{tabular}
\quad
\begin{tabular}{c}
$\abacus {
 \ci{1} &  \ci{2} & \ci{3} & \ci{4} & \ci{5} & \ci{6} \\
 \nc{7} & \nc{8} & \nc{9} & \ci{10} & \ci{11} & \nc{12}\\
}$ 
\\ \\
$\l([1, 2, 3, 6, 10, 11])=6$
\end{tabular}
\caption{The three $(3)(1)(2)$ abaci and Coxeter length of their corresponding minimal length coset representatives.}
\label{fig:312abaci}
\end{figure}

The minimal length coset representative corresponding to an $(L)(M)(R)$ abacus
having beads in positions $i$ for $n+2 \leq i \leq n+R$ together with a bead at
position $2n-M$, and gaps in positions $i$ for $n+R+1 \leq i \leq 2n-M-1$ has
Coxeter length $L+R-1$.  Notice that every time we move a bead from one of the
positions between $n+2$ and $2n-M-1$ into a gap in the position directly to its
right, the Coxeter length increases by exactly one.  In essence, we are
intertwining one sequence of length $L-1$ and one sequence of length $R-1$ and
keeping track of the number of inversions we apply.  By $q$-binomial
Interpretation \ref{interp:intertwining}, the contribution to the Coxeter length
of the minimal length coset representatives corresponding to the $(L)(M)(R)$
abaci is $q^{L+R-1}\qbinom{L+R-2}{L-1}$.
\end{proof}

For the remaining arguments, we ignore the exact entries in the base window and
simply fix both some positive number $L$ of $\L$-entries and some positive
number $R$ of $\R$-entries, and then enumerate the permutations $w\in S_n$ that
we can apply to a minimal length coset representative $w^0$ with base window of
the form $[\L\cdots\L\M\cdots\M\R\cdots\R]$.  In Theorem \ref{t:all}, we sum
the contributions over all possible values of $L$ and $R$.  

%%%%%%%%%%%%%%%%%%%%%%%%%%%%%%%%%%%%%%%%%%%%%%%
%   New Subsection
%%%%%%%%%%%%%%%%%%%%%%%%%%%%%%%%%%%%%%%%%%%%%%%

\subsubsection{Short elements with intertwining}

One possibility is that after $w\in S_n$ is applied to our minimal length coset
representative $w^0$ with base window of the form
$[\L\cdots\L\M\cdots\M\R\cdots\R]$, an $\R$-entry lies to the left of an
$\L$-entry.  In this case, we say that $w$ is {\em intertwining}, the
$\L$-entries are {\em intertwining} with the $\R$-entries, and that the
interval between the leftmost $\R$ and the rightmost $\L$ inclusive is
the {\em intertwining zone}.  

\begin{lemma}
Fix $L$ and $R > 0$.  Then, we have
\[ S_I(x,q) = \sum_w x^{n(w)} q^{\l(w)} = \hspace{4.5in}\] \[ \hspace{.5in}\sum_{M\geq 0} x^{L+M+R} \sum_{\rho = 0}^{R-1} \sum_{\lambda = 0}^{L-1} \sum_{\mu = 0}^{M} q^{Q}\qbinom{M}{\mu} \qbinom{L-\lambda-1+\mu}{\mu} \qbinom{\lambda+\rho}{\lambda} \qbinom{M-\mu+R-\rho-1}{M-\mu}, \]
where the sum on the left is over all $w\in S_{\infty}$ that are intertwining and apply to a short $(L)(M)(R)$ abacus for some $M$, and $Q=(\lambda+1)(\mu+1) + (\rho+1)(M-\mu+1) -1$.
\label{l:S_I}
\end{lemma}
\begin{proof}
By Theorem~\ref{t:321-av}, there are no $\M$-entries between the leftmost $\R$
and the rightmost $\L$, because this would create a $[321]$-pattern.  So the
$\M$-entries only occur before the leftmost $\R$ and after the rightmost $\L$.

Notice that any descent in the $\M$-entries before the leftmost $\R$ would create a
$[321]$-pattern when coupled with the rightmost $\L$.  Similarly, any descent
in the $\M$-entries after the rightmost $\L$ would create a $[321]$-pattern when
coupled with the leftmost $\R$.  Therefore, the $\M$-entries are allowed to have at
most one descent, which must occur between the $\M$-entries on either side of
the intertwining zone.  

From this, we know that the structure of $w^0 w$ is as follows.  Some number
$\lambda+1$ of $\L$-entries are intertwining with some number $\rho+1$ of
$\R$-entries in the intertwining zone.  The remaining $\L$-entries are
intertwining with some number $\mu$ of $\M$-entries on the left side of the
intertwining zone, and the remaining $\R$-entries are intertwining with $(M-\mu)$
$\M$-entries on the right side of the intertwining zone.
This structure is illustrated in Figure~\ref{fig:Intertwined}.

\begin{figure}
$[\overbrace{\L\M\cdots\L\M}^\textup{$\L$'s and $\M$'s}\overbrace{\underline{\R}\L\cdots\R\underline{\L}}^\textup{intertwining zone}\overbrace{\M\R\cdots\M\R}^\textup{$\R$'s and $\M$'s}]$
\caption{The structure of the base window of an intertwined short fully commutative element.  The leftmost $\R$ and the rightmost $\L$ are underlined.}
\label{fig:Intertwined}
\end{figure}

The contribution to the Coxeter length generating function from splitting the
$M$-entries into two sets of size $\mu$ and $M-\mu$, and transposing as
necessary in order to place the first set on the left and the right set on the
right is $\qbinom{M}{\mu}$ by Interpretation~\ref{interp:intertwining}.

Once these entries have been ordered in the minimal length configuration that
conforms to the structure shown in Figure~\ref{fig:Intertwined}, we compute the
Coxeter length offset $Q$ by counting the remaining inversions among the
entries in the base window.  We have $\mu$ $\M$-entries inverted with
$(\lambda+1)$ $\L$-entries, and $(\rho+1)$ $\R$-entries inverted with $(M-\mu)$
$\M$-entries.  In addition, the leftmost $\R$ is inverted with $\lambda$
$\L$-entries not including the rightmost $\L$, and the rightmost $\L$ is
inverted with $\rho$ $\R$-entries not including the leftmost $\R$.  Finally,
the leftmost $\R$ is inverted with the rightmost $\L$.  These inversions
contribute $Q = \mu (\lambda+1) + (M-\mu)(\rho+1) + \lambda + \rho + 1$ to the
Coxeter length.

Lastly, we can intertwine the ($L-\lambda-1$) $\L$-entries and $\mu$
$\M$-entries to the left of the zone, the $\lambda$ $\L$-entries and $\rho$
$\R$-entries in the zone, and the ($M-\mu$) $\M$-entries with the ($R-\rho-1$)
$\R$-entries to the right of the zone.  This proves the formula.
\end{proof}

%%%%%%%%%%%%%%%%%%%%%%%%%%%%%%%%%%%%%%%%%%%%%%%
%   New Subsection
%%%%%%%%%%%%%%%%%%%%%%%%%%%%%%%%%%%%%%%%%%%%%%%

\subsubsection{Short elements without intertwining}

If the $\L$-entries and $\R$-entries are not intertwined, there may be
$\M$-entries lying between the rightmost $\L$ and the leftmost $\R$.  There can
be no descents in the $\M$-entries to the left of the rightmost $\L$ nor to the
right of the leftmost $\R$ by the same reasoning as above.  However, multiple
descents may now occur among the $\M$-entries.  This structure is illustrated
in Figure \ref{fig:NonIntertwined}.  We enumerate these short elements without
intertwining by conditioning on the number of descents that occur among the
$\M$-entries.  Lemmas~\ref{l:S_0}, \ref{l:S_1}, and \ref{l:S_2} enumerate the short elements in which there are zero, one, or two or more descents among the $\M$-entries, respectively.

\begin{figure}
$[\overbrace{\L\M\cdots\M\underline{\L}}^\textup{$\L$'s and $\M$'s}\overbrace{\M\M\cdots\M\M}^\textup{$\M$ descents may occur}\overbrace{\underline{\R}\M\cdots\M\R}^\textup{$\R$'s and $\M$'s}]$
\caption{The structure of the base window of a non-intertwined short fully commutative element. The rightmost $\L$ and the leftmost $\R$ are underlined.}
\label{fig:NonIntertwined}
\end{figure}

\begin{lemma}
Fix $L$ and $R > 0$.  Then, we have
\[ S_0(x,q) = \sum_w x^{n(w)} q^{\l(w)} = \sum_{M\geq 0} x^{L+M+R} \sum_{\mu=0}^{M}q^\mu\qbinom{L-1+\mu}{\mu}\qbinom{R+M-\mu}{M-\mu}, \]
where the sum on the left is over all $w \in S_{\infty}$ that are not intertwining, have no descents among the $\M$-entries, and apply to a short $(L)(M)(R)$ abacus for some $M$.
\label{l:S_0}
\end{lemma}
\begin{proof}
Let $\mu$ be the number of $\M$-entries lying to the left of the rightmost $\L$.
Then, the $\mu$ $\M$-entries can be intertwined with the remaining $(L-1)$
$\L$-entries, and the remaining $(M-\mu)$ $\M$-entries can be intertwined with
the $R$ $\R$-entries.  

We compute the Coxeter length offset by counting the inversions among the entries in the base window in the minimal length configuration of this type.  In this case, there are simply $\mu$ $\M$-entries that are inverted with the rightmost $\L$.  Summing over all valid values of $\mu$ gives the formula.
\end{proof}

\begin{lemma}
Fix $L$ and $R > 0$.  Then, we have
\[ S_1(x,q)= \sum_w x^{n(w)} q^{\l(w)} = \sum_{M\geq 0} x^{L+M+R}\sum_{\mu=1}^{M-1}\bigg(\qbinom{M}{\mu}-1\bigg)\qbinom{L+\mu}{\mu}\qbinom{R+M-\mu}{M-\mu}, \]
where the sum on the left is over all $w \in S_{\infty}$ that are not
intertwining, have exactly one descent among the $\M$-entries, and apply to a
short $(L)(M)(R)$ abacus for some $M$.
\label{l:S_1}
\end{lemma}
\begin{proof}
Consider such permutations having a descent at the $\mu$th $\M$-entry.  The
choices for the $\M$-entries that are not the identity permutation are
enumerated by $\qbinom{M}{\mu} - 1$ by
Interpretation~\ref{interp:intertwining}.  Then, the $\M$-entries to the left of
the descent can be intertwined with the $\L$-entries, and the $\M$-entries to
the right of the descent can be intertwined with the $\R$-entries.  Summing
over all valid values of $\mu$ gives the formula.
\end{proof}

To prepare for the proof of our next result, we recall the following lemma which
solves certain generating function recurrences.

\begin{lemma}{\cite[Lemma 2.3]{MBMcolumnconvex}}\label{l:mbm_2.3}
Let $\mathcal{A}$ be the sub-algebra of the formal power series algebra
$\mathbb{R}[[s,t,x,y,q]]$ formed with series $S$ such that $S(1,t,x,y,q)$ and
$S'(1,t,x,y,q)$ are well-defined in $\mathbb{R}[[t,x,y,q]]$.
Moreover, we abbreviate $f(s,t,x,y,q) \in \mathcal{A}$ by $f(s)$.
Let $X(s,t,x,y,q)$ be a formal power series in $\mathcal{A}$.
Suppose that 
\[ X(s) = x e(s)+xf(s) X(1)+xg(s)X(sq) \]
where $e$, $f$, and $g$ are in $\mathcal{A}$.  Then, $X(s,t,x,y,q)$ is equal to
$$\frac{E(s) + E(1)F(s) - E(s)F(1)}{1-F(1)}$$
where
\[ E(s) = \sum_{n \geq 0} x^{n+1} g(s) g(sq) \cdots g(sq^{n-1}) e(sq^n) \text{ and } F(s) = \sum_{n \geq 0} x^{n+1} g(s) g(sq) \cdots g(sq^{n-1}) f(sq^n). \]
\end{lemma}

We are now in a position to enumerate the remaining elements.

\begin{lemma}
Fix $L$ and $R > 0$.  Then, we have
\[ S_2(x,q) = \sum_w x^{n(w)} q^{\l(w)} = x^{L+R}\sum_{i,j\geq 1} \qbinom{L+i}{L}\qbinom{R+j}{R}d_{i,j}(x,q), \]
where the sum on the left is over all $w \in S_{\infty}$ that are not
intertwining, have at least two descents among the $\M$-entries, and apply to a
short $(L)(M)(R)$ abacus for some $M$.  Here, $d_{i,j}(x,q)$ is the coefficient of $z^is^j$
in the generating function that satisfies the functional equation
\[ D(x,q,z,s)=\frac{\sum_{n\geq0}x^{n+1}\sum_{i=1}^{n-1}\big(\qbinom{n}{i}-1\big)z^i\big((qs)-(qs)^{n-i}\big)}{(1-qs)(1-xs)}+\frac{xqs\big(D(x,q,z,1)-D(x,q,z,qs)\big)}{(1-qs)(1-xs)} \]
and whose solution is given explicitly below.
\label{l:S_2}
\end{lemma}
\begin{proof}
In this proof, we use the ideas of Barcucci et al. \cite{barcucci},
Bousquet-M\'elou \cite{MBMcolumnconvex}, and West \cite{West}, to investigate
the structure of the permutations restricted to the $\M$'s in the
base window.  

For such a finite permutation $w\in S_M$, we consider the following statistics:  
\begin{itemize}
    \item $n(w)$ is the size of the element (represented by variable $x$), 
    \item $\l(w)$ is the number of inversions (represented by variable $q$), 
    \item $i(w)$ is the number of entries to the left of the leftmost descent (represented by variable $z$), and 
    \item $j(w)$ is the number of entries to the right of the rightmost descent (represented by variable $s$).  
\end{itemize}
Let 
\[ D(x,q,z,s) = \sum_{w} x^{n(w)} q^{\l(w)} z^{i(w)} s^{j(w)} \]
where we sum over all fully commutative permutations with at least two descents.

We require the auxiliary function $N(x,q,z,s) = \sum_{w} x^{n(w)} q^{\l(w)}
z^{i(w)} s^{j(w)}$ where we sum over all fully commutative permutations $w\in S_n$
with at least two descents such that removing the largest entry from the one-line
notation of $w$ results in a permutation that has only one descent.  Then, the
permutations counted by $N(x,q,z,s)$ are generated from fully commutative
permutations $w'$ with exactly one descent by inserting the entry $n(w')+1$ into
the one-line notation of $w'$ at some position to the right of the existing
descent in order to avoid creating a $[321]$-instance, and this creates the
second descent.  If we fix the existing descent to occur at entry $i$, then the
fully commutative permutations with exactly one descent contribute
$\qbinom{n}{i}-1$ to $N(x,q,z,s)$.  Let $k$ denote the number of entries of $w$
to the right of the position where we insert entry $n+1$.  As $k$ runs from 1 to
$n-i-1$, we have that $n$ increases by $1$, the Coxeter length $l$ increases by
$k$, there are $i$ entries to the left of the leftmost descent, and $k$ entries
to the right of the rightmost descent.  Therefore,
\[
N(x,q,z,s)=\sum_{n\geq0}\sum_{i=1}^{n-1}x^{n+1}\bigg(\qbinom{n}{i}-1\bigg)z^i
\sum_{k=1}^{n-i-1} q^{k} s^{k}
=\sum_{n\geq0}\sum_{i=1}^{n-1}x^{n+1}\bigg(\qbinom{n}{i}-1\bigg)z^i\frac{(qs)-(qs)^{n-i}}{1-qs}.
\]
We remark that $x$ divides $N(x,q,z,s)$.

Next, we have
\[ D(x,q,z,s)=N(x,q,z,s)+\sum_w \bigg(\sum_{k=1}^{j(w)}
\big(x^{n(w)+1}q^{\l(w)+k}z^{i(w)} s^{k}\big)+x^{n(w)+1} q^{\l(w)} z^{i(w)} s^{j(w)+1}\bigg) \]
where the leftmost sum is over all fully commutative permutations with at least
two descents.  This sum counts such permutations that are obtained by
inserting entry $n(w)+1$ into the one-line notation of an existing fully
commutative permutation $w$ having at least two descents, and such that
$n(w)+1$ is inserted into a position to the right of the rightmost descent.
The rightmost term corresponds to inserting into the rightmost position in the
one-line notation, while the sum from $k = 1$ to $j(w)$ corresponds to inserting
into the remaining positions in the one-line notation from right to left.  This
formula expresses a recursive construction of the permutations we are counting,
known as the generating tree.

Hence,
\[ D(x,q,z,s)=N(x,q,z,s)+\frac{xqs}{1-qs}\big(D(x,q,z,1)-D(x,q,z,qs)\big)+xsD(x,q,z,s), \]
and therefore,
\[ D(x,q,z,s)=\frac{N(x,q,z,s)}{1-xs}+\frac{xqs}{(1-qs)(1-xs)}D(x,q,z,1)+\frac{-xqs}{(1-qs)(1-xs)}D(x,q,z,qs). \] 
This functional equation has exactly the same form as those discussed in \cite{MBMcolumnconvex}; applying Lemma~\ref{l:mbm_2.3} proves that 
\[ D(x,q,z,s)=\frac{E(x,q,z,s)+E(x,q,z,1)F(x,q,z,s)-E(x,q,z,s)F(x,q,z,1)}{1-F(x,q,z,1)}, \]
where
\[ E(x,q,z,s)=\sum_{n\geq 0}x^{n+1}\frac{-qs}{(1-qs)(1-xs)}\cdots\frac{-q^ns}{(1-q^ns)(1-xq^{n-1}s)}\frac{N(x,q,z,sq^n)/x}{1-xq^ns} \]
and
\[ F(x,q,z,s)=\sum_{n\geq 0}x^{n+1}\frac{-qs}{(1-qs)(1-xs)}\cdots\frac{-q^ns}{(1-q^ns)(1-xq^{n-1}s)}\frac{q^{n+1}s}{(1-q^{n+1}s)(1-xq^ns)}.  \]
Condensing these formulas,
\[ E(x,q,z,s)=\sum_{n\geq 0}\frac{(-1)^n(sx)^nq^{\binom{n+1}{2}}}{(qs,q)_n(xs,q)_{n+1}}N(x,q,z,sq^n) \textup{ and } F(x,q,z,s)=\sum_{n\geq 0}\frac{(-1)^n(sx)^{n+1}q^{\binom{n+2}{2}}}{(qs,q)_{n+1}(xs,q)_{n+1}}. \]

The coefficient $d_{i,j}(x,q)$ of $z^i s^j$ in $D(x,q,z,s)$ enumerates the permutations applied to the $\M$'s in the base window of all sizes and lengths such that there are at least two descents and the leftmost descent is after $i$ entries and the rightmost descent is before $j$ entries.
Intertwining the $\L$-entries with the first $i$ $\M$-entries and intertwining the $\R$-entries with the last $j$ $\M$-entries gives the desired result.
\end{proof}

%%%%%%%%%%%%%%%%%%%%%%%%%%%%%%%%%%%%%%%%%%%%%%%
%   New Subsection
%%%%%%%%%%%%%%%%%%%%%%%%%%%%%%%%%%%%%%%%%%%%%%%

\subsection{Proof of the main theorem}\label{s:main_proof}

In this section, we complete the proof of our main result.

\begin{proof}[Proof of Theorem~\ref{t:all}]
Partition the set of fully commutative elements $\tld{w}$ into long elements and short elements.  The long elements in $\tld{S}_n$ are enumerated by Lemma~\ref{l:long}; we must sum over all $n$.  

Each short element $\tld{w}$ has a normalized abacus of type $(L)(M)(R)$ for some $L$, $M$, and $R$.  When this abacus is of type $(n)(0)(0)$ for some $n$, the base window for the corresponding minimal length coset representative is $[1\, 2\, \ldots\, n]$.  These elements $\tld{w}\in\tld{S}_n^{FC}$ are therefore in one-to-one correspondence with elements of $S_n^{FC}$.  Therefore, the generating function $C(x,q)$ enumerates these elements for all $n$. 

The elements that remain to be enumerated are short elements with normalized abacus of type $(L)(M)(R)$ for $R>0$.  We enumerate these elements by grouping these elements into families based on the values of $L$, $M$, and $R$.  Decompose each element $\tld{w}$ into the product of its minimal length coset representative $w^0$ and a finite permutation $w$.  Proposition~\ref{p:lmr} proves that for two minimal length coset representatives $w_1^0$ and $w_2^0$ of the same abacus type, the set of finite permutations $w$ that multiply to form a fully commutative element is the same.  Proposition~\ref{p:fixedLMR} proves that in an $(L)(M)(R)$-family of fully commutative elements, the contribution to the length from the minimal length coset representatives is $q^{L+R-1}\qbinom{L+R-2}{L-1}$.  What remains to be determined is the generating function for the contributions of the finite permutations $w$.

In an $(L)(M)(R)$-family of fully commutative elements, the finite permutations
$w$ might intermingle the $\L$ entries and the $\R$ entries of the base window
in which case there is at most one descent among the $\M$ entries at a
prescribed position; the contribution of such $w$ is given by $S_I$ in
Lemma~\ref{l:S_I}.  Otherwise, there is no intermingling and the finite
permutations $w$ may induce zero, one, or two or more descents among the $\M$
entries; these cases are enumerated by generating functions $S_0$, $S_1$, and
$S_2$ in Lemmas~\ref{l:S_0}, \ref{l:S_1}, and \ref{l:S_2}, respectively.   In
each of these lemmas, the values for $L$ and $R$ are held constant as $M$
varies.  Summing the product of the contributions of the minimal length coset
representatives and the finite permutations over all possible values of $L$ and
$R$ completes the enumeration.  
\end{proof}

%%%%%%%%%%%%%%%%%%%%%%%%%%%%%%%%%%%%%%%%%%%%%%%
%   New Section
%%%%%%%%%%%%%%%%%%%%%%%%%%%%%%%%%%%%%%%%%%%%%%%

\section{Numerical Conclusions}
\label{s:numerical}
Theorem~\ref{t:all} allows us to determine the length generating function
$f_n(q)$ for the fully commutative elements of $\tld{S}_n$ as $n$ varies.  The
first few series $f_n(q)$ are presented below.

\medskip
\begin{tabular}{l}
$f_3(q)=1 + 3 q + {\bf 6 q^2 + 6 q^3 + 6 q^4}+\cdots$ \\
$f_4(q)=1 + 4 q + 10 q^2 + {\bf 16 q^3 + 18 q^4 + 16 q^5 + 18q^6}+\cdots$ \\
$f_5(q)=1 + 5 q + 15 q^2 + 30 q^3 + 45 q^4 + {\bf 50 q^5 + 50 q^6 + 50 q^7 + 50 q^8 + 50 q^9 }+\cdots$ \\
$f_6(q)=1 + 6 q + 21 q^2 + 50 q^3 + 90 q^4 + 126 q^5 + 146 q^6 + {\bf 150 q^7 + 156 q^8 + 152 q^9 +}$ \\ 
$\qquad\qquad{\bf 156 q^{10} + 150 q^{11} + 158 q^{12} + 150 q^{13} + 156 q^{14} + 152 q^{15}}+\cdots$ \\
$f_7(q)=1 + 7 q + 28 q^2 + 77 q^3 + 161 q^4 + 266 q^5 + 364 q^6 + 427 q^7 + 462 q^8 + 483 q^9 + {\bf 490 q^{10} + }$ \\ 
$\qquad\qquad{\bf 490 q^{11} + 490 q^{12} + 490 q^{13} + 490 q^{14} + 490 q^{15}}+\cdots$ \\
$f_8(q)=1 + 8 q + 36 q^2 + 112 q^3 + 266 q^4 + 504 q^5 + 792 q^6 + 1064 q^7 + 1274 q^8 + 1416 q^9 + $\\
$\qquad\qquad 1520 q^{10} + 1568 q^{11} + 1602 q^{12} + {\bf 1600 q^{13} + 1616 q^{14} + 1600 q^{15} + 1618 q^{16} + }$ \\
$\qquad\qquad {\bf 1600 q^{17} + 1616 q^{18} + 1600 q^{19} + 1618 q^{20}}+\cdots$ \\
$f_9(q)=1 + 9 q + 45 q^2 + 156 q^3 + 414 q^4 + 882 q^5 + 1563 q^6 + 2367 q^7 + 3159 q^8 + 3831 q^9 + $ \\
$\qquad\qquad 4365 q^{10} + 4770 q^{11} + 5046 q^{12} + 5220 q^{13} + 5319 q^{14} + 5370 q^{15} + 5391 q^{16} + {\bf 5400 q^{17} +}$ \\
$\qquad\qquad{\bf 5406 q^{18} + 5400 q^{19} + 5400 q^{20} + 5406 q^{21} + 5400 q^{22} + 5400 q^{23}}+\cdots$ \\
$f_{10}(q)=1 + 10 q + 55 q^2 + 210 q^3 + 615 q^4 + 1452 q^5 + 2860 q^6 + 4820 q^7 + 7125 q^8 + 9470 q^9 +$ \\
$\qquad\qquad 11622 q^{10} + 13470 q^{11} + 15000 q^{12} + 16160 q^{13} + 17030 q^{14} + 17602 q^{15} + 18010 q^{16} +$ \\
$\qquad\qquad  18210 q^{17} + 18380 q^{18} + 18410 q^{19} + 18482 q^{20} + {\bf 18450 q^{21} + 18500 q^{22} + 18450 q^{23} + }$ \\
$\qquad\qquad {\bf 18500 q^{24} + 18452 q^{25} + 18500 q^{26} + 18450 q^{27} + 18500 q^{28} + 18450 q^{29} + }$ \\
$\qquad\qquad {\bf 18502 q^{30} + 18450 q^{31} + 18500 q^{32} + 18450 q^{33} + 18500 q^{34} + 18452 q^{35}}+\cdots$ \\
$f_{11}(q)=1 + 11 q + 66 q^2 + 275 q^3 + 880 q^4 + 2277 q^5 + 4928 q^6 + 9141 q^7 + 14850 q^8 + 21571 q^9 + $\\
$\qquad\qquad 28633 q^{10} + 35453 q^{11} + 41690 q^{12} + 47135 q^{13} + 51667 q^{14} + 55297 q^{15} + 58091 q^{16} + $ \\
$\qquad\qquad 60159 q^{17} + 61622 q^{18} + 62623 q^{19} + 63272 q^{20} + 63668 q^{21} + 63910 q^{22} + 64031 q^{23} + $ \\
$\qquad\qquad  64086 q^{24} + 64119 q^{25} + {\bf 64130 q^{26} + 64130 q^{27} + 64130 q^{28} + 64130 q^{29}}+\cdots$ \\
$f_{12}(q)=1 + 12 q + 78 q^2 + 352 q^3 + 1221 q^4 + 3432 q^5 + 8086 q^6 + 16356 q^7 + 28974 q^8 + $\\
$\qquad\qquad 45772 q^9 + 65670 q^{10} + 87120 q^{11} + 108690 q^{12} + 129288 q^{13} + 148170 q^{14} + 164776 q^{15} +$ \\
$\qquad\qquad  178980 q^{16} + 190680 q^{17} + 200148 q^{18} + 207444 q^{19} + 213084 q^{20} + 217096 q^{21} + $ \\
$\qquad\qquad  220098 q^{22} + 222012 q^{23} + 223458 q^{24} + 224172 q^{25} + 224814 q^{26} + 224992 q^{27} + $ \\
$\qquad\qquad  225276 q^{28} + 225216 q^{29} + 225408 q^{30} + {\bf 225264 q^{31} + 225420 q^{32} + 225280 q^{33} +} $ \\
$\qquad\qquad {\bf 225414 q^{34} + 225264 q^{35} + 225438 q^{36} + 225264 q^{37} + 225414 q^{38} + 225280 q^{39} +} $ \\
$\qquad\qquad {\bf  225420 q^{40} + 225264 q^{41} + 225432 q^{42} + 225264 q^{43} + 225420 q^{44} + 225280 q^{45} +} $ \\
$\qquad\qquad {\bf  225414 q^{46} + 225264 q^{47} + 225438 q^{48} + 225264 q^{49} + 225414 q^{50}}+\cdots$
\end{tabular}

\medskip
One remarkable quality of these series is their periodicity, given by the bold-faced terms.  This behavior is explained by the following corollary to Lemma~\ref{l:long}.

\begin{corollary}
The coefficients $a_i$ of $f_n(q) = \sum_{w \in \tld{S}_n^{FC}} q^{\l(w)} =
\sum_{i \geq 0} a_i q^i$ are periodic with period $m|n$ for sufficiently large
$i$.  When $n = p$ is prime, the period $m = 1$ and in this case there
are precisely
\[ {1 \over {p}} \Bigg( { {2p} \choose p} - 2 \Bigg) \]
fully commutative elements of length $i$ in $\widetilde{S}_p$, when $i$ is
sufficiently large.
\label{c:period}
\end{corollary}
\begin{proof}
For a given $n$, the number of short fully commutative elements is finite.  The
formula for long elements in Lemma \ref{l:long} is a polynomial divided by
$1-q^n$.  Hence, the coefficients of this generating function satisfy $a_{i+n}
= a_i$, by a fundamental result on rational generating functions.

We have a factor of $(1-q^n)=(1-q)(1+q+\cdots+q^{n-1})$ in the numerator of
$\qbinom{n}{k}$ and when $n$ is prime, $(1+q+\cdots+q^{n-1})$ is irreducible.
Therefore $\qbinom{n}{k}$ contains a factor of $(1+q+\cdots+q^{n-1})$ for every
$k$ between $1$ and $n-1$.  Factoring one copy out of the sum in the expression
of Lemma \ref{l:long} and canceling with the same factor in the denominator of
$\frac{q^n}{1-q^n}$ leaves a denominator of $(1-q)$.  

Hence, we have that
\[ P(q) := { {q^n} \over {1 + q + \cdots + q^{n-1}} } \sum_{k = 1}^{n-1} \qbinom{n}{k}^{\,2} \]
is the polynomial numerator of the rational generating function $P(q) / (1-q)$
for the number of long fully commutative elements.  Therefore, when $i$ is
larger than the degree of $P(q)$, the coefficient of $q^i$ in the series
expansion of $P(q) / (1-q)$ is $P(1)$.  After substituting $q=1$ and applying Vandermonde's identity, 
\[ P(1)= \frac{1}{n}\sum_{k=1}^{n-1} \binom{n}{k}^2=\frac{1}{n}\Bigg(\sum_{k=0}^{n} \binom{n}{k}^2-2\Bigg)=\frac{1}{n}\Bigg(\binom{2n}{n}-2\Bigg),\]  as desired. 
\end{proof}

The distinction between long and short elements allows us to enumerate the
fully-commutative elements efficiently.  In some respects, this division is not
the most natural in that the periodicity of the above series begins before
there exist no more short elements.  Experimentally, it appears that the
periodicity begins at $1+\big\lfloor (n-1)/2\big\rfloor \big\lceil (n-1)/2
\big\rceil$; whereas, we can prove that the longest short element has length
$2\big\lfloor n/2\big\rfloor \big\lceil n/2 \big\rceil$.  

We begin by bounding the Coxeter length of finite fully commutative permutations.

\begin{definition}\label{d:bigr}
If $w$ has a unique left descent and $w$ has a unique right descent, then we say
that $w$ is \em bi-Grassmannian\em.
\end{definition}

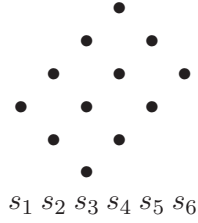
\begin{figure}
\begin{center}
\begin{tabular}{l}
\xymatrix @=-4pt @! {
 \hs & \hs & \hs & \hf & \hs & \hs \\
 \hs & \hs & \hf & \hs & \hf & \hs \\
 \hs & \hf & \hs & \hf & \hs & \hf \\
 \hf & \hs & \hf & \hs & \hf & \hs \\
 \hs & \hf & \hs & \hf & \hs & \hs \\
 & \hs & \hf & \hs & \hs & \hs \\
 s_1 & s_2 & s_3 & s_4 & s_5 & s_6 \\
}
\end{tabular}
\end{center}
\caption{The heap diagram of a bi-Grassmannian permutation $w = s_3 s_2 s_4 s_1
s_3 s_5 s_2 s_4 s_6 s_3 s_5 s_4$}
\label{fig:bigr}
\end{figure}

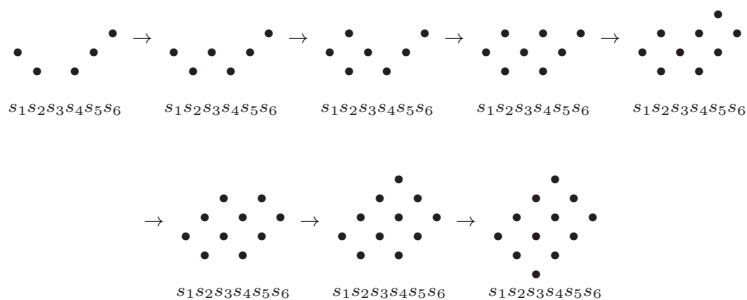
\begin{figure}
\begin{center}
\begin{tabular}{c}
{ \tiny
\xymatrix @=-7pt @! {
 \hs & \hs & \hs & \hs & \hs & \hs \\
 \hs & \hs & \hs & \hs & \hs & \hs \\
 \hs & \hs & \hs & \hs & \hs & \hf \\
 \hf & \hs & \hs & \hs & \hf & \hs \\
 \hs & \hf & \hs & \hf & \hs & \hs \\
 & \hs & \hs & \hs & \hs & \hs \\
 s_1 & s_2 & s_3 & s_4 & s_5 & s_6 \\
}
\parbox{0.1in} { \vspace{0.5in} $\rightarrow$ }
\xymatrix @=-7pt @! {
 \hs & \hs & \hs & \hs & \hs & \hs \\
 \hs & \hs & \hs & \hs & \hs & \hs \\
 \hs & \hs & \hs & \hs & \hs & \hf \\
 \hf & \hs & \hf & \hs & \hf & \hs \\
 \hs & \hf & \hs & \hf & \hs & \hs \\
 & \hs & \hs & \hs & \hs & \hs \\
 s_1 & s_2 & s_3 & s_4 & s_5 & s_6 \\
}
\parbox{0.1in} { \vspace{0.5in} $\rightarrow$ }
\xymatrix @=-7pt @! {
 \hs & \hs & \hs & \hs & \hs & \hs \\
 \hs & \hs & \hs & \hs & \hs & \hs \\
 \hs & \hf & \hs & \hs & \hs & \hf \\
 \hf & \hs & \hf & \hs & \hf & \hs \\
 \hs & \hf & \hs & \hf & \hs & \hs \\
 & \hs & \hs & \hs & \hs & \hs \\
 s_1 & s_2 & s_3 & s_4 & s_5 & s_6 \\
}
\parbox{0.1in} { \vspace{0.5in} $\rightarrow$ }
\xymatrix @=-7pt @! {
 \hs & \hs & \hs & \hs & \hs & \hs \\
 \hs & \hs & \hs & \hs & \hs & \hs \\
 \hs & \hf & \hs & \hf & \hs & \hf \\
 \hf & \hs & \hf & \hs & \hf & \hs \\
 \hs & \hf & \hs & \hf & \hs & \hs \\
 & \hs & \hs & \hs & \hs & \hs \\
 s_1 & s_2 & s_3 & s_4 & s_5 & s_6 \\
}
\parbox{0.1in} { \vspace{0.5in} $\rightarrow$ }
\xymatrix @=-7pt @! {
 \hs & \hs & \hs & \hs & \hs & \hs \\
 \hs & \hs & \hs & \hs & \hf & \hs \\
 \hs & \hf & \hs & \hf & \hs & \hf \\
 \hf & \hs & \hf & \hs & \hf & \hs \\
 \hs & \hf & \hs & \hf & \hs & \hs \\
 & \hs & \hs & \hs & \hs & \hs \\
 s_1 & s_2 & s_3 & s_4 & s_5 & s_6 \\
}
} \\ 
{ \tiny
\parbox{0.1in} { \vspace{0.5in} $\rightarrow$ }
\xymatrix @=-7pt @! {
 \hs & \hs & \hs & \hs & \hs & \hs \\
 \hs & \hs & \hf & \hs & \hf & \hs \\
 \hs & \hf & \hs & \hf & \hs & \hf \\
 \hf & \hs & \hf & \hs & \hf & \hs \\
 \hs & \hf & \hs & \hf & \hs & \hs \\
 & \hs & \hs & \hs & \hs & \hs \\
 s_1 & s_2 & s_3 & s_4 & s_5 & s_6 \\
}
\parbox{0.1in} { \vspace{0.5in} $\rightarrow$ }
\xymatrix @=-7pt @! {
 \hs & \hs & \hs & \hf & \hs & \hs \\
 \hs & \hs & \hf & \hs & \hf & \hs \\
 \hs & \hf & \hs & \hf & \hs & \hf \\
 \hf & \hs & \hf & \hs & \hf & \hs \\
 \hs & \hf & \hs & \hf & \hs & \hs \\
 & \hs & \hs & \hs & \hs & \hs \\
 s_1 & s_2 & s_3 & s_4 & s_5 & s_6 \\
}
\parbox{0.1in} { \vspace{0.5in} $\rightarrow$ }
\xymatrix @=-7pt @! {
 \hs & \hs & \hs & \hf & \hs & \hs \\
 \hs & \hs & \hf & \hs & \hf & \hs \\
 \hs & \hf & \hs & \hf & \hs & \hf \\
 \hf & \hs & \hf & \hs & \hf & \hs \\
 \hs & \hf & \hs & \hf & \hs & \hs \\
 & \hs & \hf & \hs & \hs & \hs \\
 s_1 & s_2 & s_3 & s_4 & s_5 & s_6 \\
}
}
\end{tabular}
\end{center}
\caption{The construction of a bi-Grassmannian permutation containing $w = s_2 s_4 s_1
s_5 s_6$}
\label{fig:bires}
\end{figure}

\begin{lemma}\label{l:bigr}
Suppose $\mathsf{w}$ is a reduced expression for $w \in S_n^{FC}$.  Then there
exists a bi-Grassmannian permutation $x$ with reduced expression $\mathsf{x} =
\mathsf{u} \mathsf{w} \mathsf{v}$.  In particular, $\l(w) \leq \big\lfloor
n/2\big\rfloor \big\lceil n/2 \big\rceil$.
\end{lemma}
\begin{proof}
Recall the coalesced heap diagram from \cite[Section 3]{b-w} associated to any
fully commutative element $w$.  This diagram is an embedding of the Hasse
diagram of the heap poset defined in \cite{s1} into $\mathbb{Z}^2$.  In this
diagram, an entry of the heap poset represented by $(x,y) \in \mathbb{Z}^2$ is
labeled by the Coxeter generator $s_i$ if and only if $x = i$.  Moreover, 
we have that a generator represented by $(x,y)$ covers a generator represented by
$(x',y')$ in the heap poset if and only if $y = y' + 1$ and $x = x' \pm 1$.  See
\cite[Remark 5]{b-w} for details.  An example of a heap diagram is shown in
Figure~\ref{fig:bigr}.

Next, we describe a sequence of length-increasing multiplications on the left
and right that will transform $w$ into a bi-Grassmannian permutation.  An
example of this construction is illustrated in Figure~\ref{fig:bires}.  First,
if there are any columns $1 \leq i \leq (n-1)$ in the heap diagram of $w$ that
do not contain an entry, then multiply on the right by $s_i$ to add an entry to
the heap diagram, and then recoalesce the heap diagram.  Henceforth, we assume
that every column in the heap diagram of $w$ has at least one entry.  Moreover,
it follows from \cite[Lemma 1]{b-w} that columns $1$ and $(n-1)$ of the heap
diagram contain precisely one entry.

Next, consider the \em ridgeline \em in the heap diagram of $w$ consisting of the
points that correspond to maximal elements in the heap poset.  By construction,
the ridgeline can be interpreted as a lattice path consisting of \em up-steps
\em of the form $\big( (i,y), (i+1,y+1) \big)$ and \em down-steps \em of the
form $\big( (i,y), (i+1,y-1) \big)$.  For each sequence of the form $\big(
(i,y), (i+1,y-1), (i+2,y) \big)$, we multiply on the right by an $s_{i+1}$
generator to add a new entry to the ridgeline and transform the sequence from
down-up to up-down.  When we have performed these multiplications until there
are no more down-up sequences along the ridgeline, our heap diagram encodes a
fully-commutative permutation with a unique right descent.  In a completely
similar fashion, we can also perform multiplications on the left to produce a
heap which encodes a fully-commutative permutation with a unique left descent.
Hence, our transformed permutation is bi-Grassmannian.  

When $w$ is bi-Grassmannian, the heap of $w$ forms a quadrilateral by
\cite[Lemma 1]{b-w} as illustrated in Figure~\ref{fig:bigr}.  The Coxeter length
of $w$ is the number of lattice points in the quadrilateral, and this is
maximized when the unique left and right descents occur as close to $n/2$ as
possible.  Hence, $\l(w) \leq \big\lfloor n/2\big\rfloor \big\lceil n/2
\big\rceil$. 
\end{proof}

\begin{proposition}
Let $w \in \tld{S}_n^{FC}$ be a short element.  Then $\l(w) \leq 2\big\lfloor n/2\big\rfloor \big\lceil n/2 \big\rceil$.  In addition, there exists a $w \in \tld{S}_n^{FC}$ with length $2\big\lfloor n/2\big\rfloor \big\lceil n/2 \big\rceil$.
\label{p:longshort}
\end{proposition}
\begin{proof}
Let $\tld{w} \in \tld{S}_n^{FC}$ be a short element.  Then, by the
parabolic decomposition, $\tld{w} = w^0 w$ where $w \in S_n^{FC}$ and $w^0$ is a
minimal length coset representative with an associated $(L)(M)(R)$ abacus.

First, we determine the values of $L$, $M$, and $R$ that give $w^0\in
\tld{S}_n/S_n$ of longest Coxeter length.
For fixed $L$, $M$, and $R$, Proposition~\ref{p:length} implies that the
longest Coxeter length of a minimal length coset representative having an
$(L)(M)(R)$ abacus occurs when there are gaps in positions $n+1$ through $n+L$,
beads in positions $n+L+1$ through $n+L+R$ and gaps in positions $n+L+R+1$
through $2n$.  The length of this minimal length coset representative is $LR$,
which is maximized when $M=0$ and $L$ and $R$ are as close to $n/2$ as possible.  
Therefore, $\l(w^0) \leq \big\lfloor n/2\big\rfloor \big\lceil n/2 \big\rceil$ and is exactly equal in the case where $w^0=\big[1,2,\ldots,\lfloor n/2\big\rfloor,n+\lfloor n/2\big\rfloor+1,\ldots,2n\big]$.

Considering the other factor $w \in S_n^{FC}$, it follows from Lemma~\ref{l:bigr}
that $\l(w) \leq \big\lfloor n/2\big\rfloor \big\lceil n/2 \big\rceil$.
Moreover, this length is maximized when $w$ is a bi-Grassmannian permutation
with left and right descents occurring as close to $n/2$ as possible.

Adding the bounds to obtain $\l(\tld{w}) = \l(w) + \l(w^0) \leq 2 \big\lfloor
n/2\big\rfloor \big\lceil n/2 \big\rceil$ proves the result.  In addition, the
one-line notation for a bi-Grassmannian permutation has the form
$[i+1,i+2,\ldots,n,1,2,\ldots,i]$ for some $i$.  When $i=\lfloor
n/2\big\rfloor$, this bi-Grassmannian applies directly to the above affine
permutation to give the fully commutative affine permutation $\big[n+\lfloor
n/2\big\rfloor+1,\ldots,2n,1,2,\ldots,\lfloor n/2\big\rfloor\big]$ of length $2
\big\lfloor n/2\big\rfloor \big\lceil n/2 \big\rceil$.
\end{proof}

Corollary~\ref{c:period} and Proposition \ref{p:longshort} give another way to
compute the series $f_{n}(q)$, without invoking Theorem~\ref{t:all}.  Using
a computer program, one needs simply to count the fully-commutative elements of
$\tld{S}_n$ of length up to $n+2\big\lfloor n/2\big\rfloor \big\lceil n/2
\big\rceil$.

\bigskip
\section{Further questions}
\label{s:future}

In this work, we have studied the length generating function for the fully
commutative affine permutations.  It would be interesting to explore the
ramifications of the periodic structure of these elements in terms of the
affine Temperley--Lieb algebra.  Also, all of our work should have natural
extensions to the other Coxeter groups.  In fact, we know of no analogue of
\cite{barcucci} enumerating the fully commutative elements by length for finite
types beyond type $A$.  It is a natural open problem to establish the
periodicity of the length generating functions for the other affine types.  It
would also be interesting to determine the analogues for other types of the
$q$-binomial coefficients and $q$-Bessel functions that played prominent roles
in our enumerative formulas.

Finally, it remains an open problem to prove that the periodicity of the length
generating function coefficients for fixed rank begins at length $1+\big\lfloor
(n-1)/2\big\rfloor \big\lceil (n-1)/2 \big\rceil$, as indicated by the data.
By examining the structure of the heap diagrams associated to the fully
commutative affine permutations, we have discovered some plausible reasoning
indicating this tighter bound, but a proof remains elusive.

%%%%%%%%%%%%%%%%%%%%%%%%%%%%%%%%%%%%%%%%%%%%%%%%%%%%%%%%%%%%%%%%%%%%%
%  Section
%%%%%%%%%%%%%%%%%%%%%%%%%%%%%%%%%%%%%%%%%%%%%%%%%%%%%%%%%%%%%%%%%%%%%

\bigskip
\section*{Acknowledgments}

We are grateful to the anonymous referees for their insightful comments, and for
suggesting the formula in Corollary~\ref{c:period} for the asymptotic number of
affine fully commutative permutations with a given length in prime rank.

%%%%%%%%%%%%%%%%%%%%%%%%%%%%%%%%%%%%%%%%%%%%%%%%%%%%%%%%%%%%%%%%%%%%%
% ==========   Bibliography
%%%%%%%%%%%%%%%%%%%%%%%%%%%%%%%%%%%%%%%%%%%%%%%%%%%%%%%%%%%%%%%%%%%%% 

\bigskip
%\bibliographystyle{alpha}
%\bibliography{affine_fc_arxiv_rev}

\end{document}